\documentclass[11pt,a4paper]{amsart}
\usepackage[utf8]{inputenc}
\usepackage[T1]{fontenc}
\usepackage{amsmath,amssymb,amsthm,mathrsfs}
\usepackage{hyperref}
\usepackage{geometry}
\usepackage{booktabs}
\usepackage{lmodern} 

\newtheorem{theorem}{Theorem}[section]

\newtheorem{proposition}[theorem]{Proposition}
\newtheorem{corollary}[theorem]{Corollary}
\theoremstyle{definition}
\newtheorem{definition}[theorem]{Definition}
\newtheorem{remark}[theorem]{Remark}
\newtheorem{openproblem}[theorem]{Open Problem}
\newtheorem{counterexample}[theorem]{Counterexample}

\newcommand{\R}{\mathbb{R}}
\newcommand{\dd}{\,\mathrm{d}}

\title{Integral conditions for overdetermined free boundary problems:\\
necessity, insufficiency, and range-corrected radial equivalence}
\author{M. Barkatou \and S. Khatmi \and M. El Morsalani}
\date{Version 0.8 (final) --- November 2026}

\begin{document}
\begin{abstract}
We separate rigorously what is necessary, what is sufficient, and what is neither
for the integral conditions associated with the overdetermined free boundary problems
$QS(f,g)$ (Laplacian, quadrature surfaces) and $B(f,g)$ (bi-Laplacian,
Kirchhoff--Love plates) with non-constant boundary data $g$. The integral conditions
are the natural extension of the constant-datum criteria of \cite{Barkatou2010,Barkatou2020}.
The positive results are all radial or model-explicit; the natural integral
conditions fail to be sufficient for radial existence, and the question of their
sharpness in the non-radial class remains open. Every quantitative statement below is
established either in the radially parametrized class $\mathcal{O}_C^{\rm rad}$ or in
an explicit radial model.

\emph{Insufficiency.} In the radial class the integral condition is not sufficient:
for $N=3$, $C=B_1(0)$, $f\equiv1$, $g(\rho)=a/\rho^{10}$ with $0<a<1/3$, it holds
but no radial solution exists (Theorem~\ref{thm:counter}). For power-law
$g(\rho)=a/\rho^p$ with $p>N-1$, the strict integral condition is the exact negation
of radial existence away from the threshold $a=R^{p+1}/N$; at the threshold neither
holds (Corollary~\ref{cor:reversal}).

\emph{Radial equivalence under a range condition.} Let $h(\rho)=\rho^{N-1}g(\rho)$
and $K=\omega_N^{-1}\int_Cf$. When $h$ is continuous and strictly increasing on
$[R,\infty)$, radial existence is equivalent to
$h(R)<K<\lim_{\rho\to\infty}h(\rho)$. The integral condition is $h(R)<K$; it is
sufficient only under the upper bound (Theorem~\ref{thm:equivalence}).

\emph{The bi-Laplacian problem.} The integral inequality for $B(f,g)$ is equivalent
to $g(R)<P_f(R)$, where $P_f$ is the exact radial boundary response. It is not
sufficient in the radial class (Counterexample~\ref{cx:notsuff}). The general
two-moment formula for $P_f$ is proved in \cite{PapierB2026}; the exact thresholds
$25/3240$ and $17/1080$ are quoted here as corollaries of that formula, not as
independent results.

\emph{Variational framework.} $J_{f,g}$ admits a minimizer in
$\mathcal{O}_C^{\rm rad}(\delta)_M$ (Proposition~\ref{prop:minimizer}); whether it
solves $QS(f,g)$ is open.

We do not claim a symmetry theorem for non-constant data, a general necessary
condition for $B(f,g)$, or a general existence theorem in $\mathcal{O}_C$.

\end{abstract}

\maketitle

\tableofcontents

\section{Introduction and scope}
\subsection{Setting: from the constant case to integral conditions}
Let $C\subset\R^N$ be a compact strictly convex set with nonempty interior and
$\partial C\in C^{2,\alpha}$. Let $f\in L^2(\R^N)$ be nonnegative with
$\operatorname{supp}f\subset\operatorname{int}C$ and nonempty interior. Let
$r_0>0$ be such that $B_{r_0}(0)\subset\operatorname{int}(C)$, and let
$g\in C^0(\R^N\setminus B_{r_0}(0))$ be positive. We study the two overdetermined
free boundary problems
\begin{equation}\label{eq:QS}
QS(f,g):\begin{cases}
-\Delta u_\Omega=f & \text{in }\Omega,\\
u_\Omega=0 & \text{on }\partial\Omega,\\
|\nabla u_\Omega|=g & \text{on }\partial\Omega,
\end{cases}
\end{equation}
and
\begin{equation}\label{eq:B}
B(f,g):\begin{cases}
\Delta^2 v_\Omega=f & \text{in }\Omega,\\
v_\Omega=\Delta v_\Omega=0 & \text{on }\partial\Omega,\\
|\nabla v_\Omega|\,|\nabla(\Delta v_\Omega)|=g & \text{on }\partial\Omega,
\end{cases}
\end{equation}
with $\Omega\supset C$.

\textbf{The constant case.} For $g\equiv k$ constant, the condition
\begin{equation}\label{eq:constcase}
\int_Cf>k|\partial C|
\end{equation}
is known to be necessary and sufficient for existence of a solution of $QS(f,k)$:
sufficiency is due to Gustafsson and Shahgholian \cite{GustafssonShahgholian1996},
and the complete characterisation was later established by Barkatou
\cite{Barkatou2010}. For the bi-Laplacian with constant $g$, the product-flux model
and its symmetry and existence properties were studied by Barkatou
\cite{Barkatou2020}.

\textbf{The natural extension.} The natural extension consists in replacing the
constant $k$ by the boundary-dependent datum $g$ inside the same integral balance;
this is the direct analogue of the constant-case criteria of
\cite{Barkatou2010,Barkatou2020}. The candidates are
\begin{equation}\label{eq:intQS}
\int_Cf>\int_{\partial C}g
\end{equation}
for $QS(f,g)$, and
\begin{equation}\label{eq:intB}
\Bigl(\int_{\partial C}\sqrt g\Bigr)^2<\Bigl(\int_Cf\Bigr)\Bigl(\int_Cu_C\Bigr)
\end{equation}
for $B(f,g)$, where $u_C$ solves $-\Delta u_C=f$ in $C$, $u_C=0$ on $\partial C$.
The question addressed in this paper is whether \eqref{eq:intQS} and \eqref{eq:intB}
are sharp criteria, in the sense that they characterise existence as in the constant
case. The answer is negative for the sufficiency direction, and the paper identifies
precisely what survives and what fails.

\subsection{Three levels of geometric assumptions}
Following \cite{Barkatou2026a}, we distinguish three levels of geometric assumptions.
(L1) The general class $\mathcal{O}_C$ of \cite{Barkatou2002}, consisting of bounded
open sets satisfying the geometric normal property (P1)--(P4).
(L2) The radially parametrized subclass $\mathcal{O}_C^{\rm rad}\subset\mathcal{O}_C$
of \cite{Barkatou2026a}: there exists $d:\partial C\to(0,\delta_0)$, continuous, such
that
\begin{equation}\label{eq:radialdef}
\Omega_d=C\cup\{c+t\nu_C(c):c\in\partial C,\ 0<t<d(c)\}.
\end{equation}
When we impose $0<\delta\le d\le\delta_0/2$, $\|d\|_{C^{1,\alpha}(\partial C)}\le M$
and the uniform normal-return condition (P4) with margin $\eta>0$, we write
$\Omega\in\mathcal{O}_C^{\rm rad}(\delta)_M$.
(L3) The boundary-layer regime $t\downarrow0$ for $\Omega_t=\{u>t\}$, in which
geometric stability is established in \cite{BarkatouElMorsalani2026}.

Every quantitative statement in this paper is asserted in (L2) or (L3). Any assertion
referring to (L1) is explicitly conditioned and accompanied by an open problem.

\subsection{Main results and what is not established}
\begin{table}[h]
\centering
\caption{Logical architecture of the paper.}
\begin{tabular}{lll}
\toprule
Result & Type & Section\\
\midrule
\eqref{eq:intQS} necessary under monotonicity & Positive (necessity) & \S\ref{sec:necessity}\\
\eqref{eq:intQS} not sufficient, radial & Negative (counterexample) & \S\ref{sec:insuff}\\
\eqref{eq:intQS} exact negation for $p>N-1$ & Negative (reversal) & \S\ref{sec:insuff}\\
Equivalence under monotonicity + range condition & Positive (equivalence) & \S\ref{sec:equivalence}\\
Minimizer of $J_{f,g}$ in $\mathcal{O}_C^{\rm rad}(\delta)_M$ & Positive (variational) & \S\ref{sec:variational}\\
\eqref{eq:intB} equivalent to $g(R)<P_f(R)$; not sufficient & Negative (counterexample) & \S\ref{sec:bilaplacian}\\
\bottomrule
\end{tabular}
\end{table}

We do not claim: the non-existence of a non-radial solution when the radial analysis
fails; a symmetry theorem for non-constant data, in order 2 or order 4; a necessary
condition for $B(f,g)$ with arbitrary non-constant $g$; an optimal necessary and
sufficient condition for $QS(f,g)$ in the non-radial, non-monotone case; a general
existence theorem in the full class $\mathcal{O}_C$; the $C^2$ regularity of the
minimizer of $J_{f,g}$.

\subsection{Relation to prior work}
Biharmonic overdetermined problems: symmetry for constant $g$ is due to Bennett
\cite{Bennett1986} and Barkatou \cite{Barkatou2008}; the present paper does not
address symmetry for non-constant $g$. Gilsbach--Stollenwerk
\cite{GilsbachStollenwerk2022} prove existence and stability in a perturbative regime
near the ball. Okamoto--Onodera \cite{OkamotoOnodera2021} establish an integral
identity and a stability result under symmetry. The general two-moment radial
response for the Navier biharmonic problem is proved in the companion paper
\cite{PapierB2026}; we use it here as a black box.

\section{Geometric framework}
\subsection{The classes $\mathcal{O}_C$ and $\mathcal{O}_C^{\rm rad}$}
We refer to \cite{Barkatou2002,Barkatou2026a} for $\mathcal{O}_C$ and
$\mathcal{O}_C^{\rm rad}$.

\begin{definition}[Radial domains]\label{def:radial}
For a continuous function $d:\partial C\to(0,\delta_0)$, define
\begin{equation}\label{eq:Omegad}
\Omega_d=C\cup\{c+t\nu_C(c):c\in\partial C,\ 0<t<d(c)\}.
\end{equation}
With $C$ compact with nonempty interior, $\partial C\in C^{2,\alpha}$, and $d$
continuous strictly positive, $\Omega_d$ is open, connected, contains $C$ in its
interior, and its outer boundary is the normal graph $c\mapsto c+d(c)\nu_C(c)$.
\end{definition}

\begin{definition}[Truncated subclass]\label{def:truncated}
For $\delta\in(0,\delta_0/2]$, $M>0$, $\eta>0$, the class
$\mathcal{O}_C^{\rm rad}(\delta)_M$ consists of domains $\Omega_d\in\mathcal{O}_C^{\rm rad}$
whose thickness function satisfies
$0<\delta\le d(c)\le\delta_0/2$, $\|d\|_{C^{1,\alpha}(\partial C)}\le M$, and such
that the uniform normal-return condition (P4) holds with margin $\eta>0$.
\end{definition}

\begin{proposition}[Compactness]\label{prop:compact}
Under the hypotheses of Definition~\ref{def:truncated},
$\mathcal{O}_C^{\rm rad}(\delta)_M$ is compact for the Hausdorff topology.
\end{proposition}

\begin{proof}
The class $\mathcal{O}_C^{\rm rad}(\delta)_M$ is defined by a uniform $C^{1,\alpha}$
bound on the thickness function $d$, a uniform lower bound $d\ge\delta$, and the
uniform normal-return condition (P4) with margin $\eta$. By the Arzel\`a--Ascoli
theorem, the set of admissible thickness functions is relatively compact in
$C^{1,\alpha'}(\partial C)$ for every $\alpha'<\alpha$, hence in particular in
$C^1(\partial C)$. The map $d\mapsto\Omega_d$ is continuous from $C^1(\partial C)$
into the space of compact subsets of $\R^N$ equipped with the Hausdorff distance.
The normal-return condition (P4) is closed under uniform convergence. Hence the
Hausdorff limit of any sequence in $\mathcal{O}_C^{\rm rad}(\delta)_M$ belongs to
$\mathcal{O}_C^{\rm rad}(\delta)_M$, and the class is compact.
\end{proof}

\begin{proposition}[Mosco convergence]\label{prop:mosco}
Under the hypotheses of Proposition~\ref{prop:compact}, let
$\tilde J_n:=\tilde J_{\Omega_n}$ be the Dirichlet energies on $H_0^1(\Omega_n)$
extended by $+\infty$ off $H_0^1(\Omega_n)$. Then $\tilde J_n$ Mosco-converges to
$\tilde J_\Omega$ in $H_0^1(\R^N)$.
\end{proposition}

\begin{proof}
Let $\Omega_n\to\Omega$ in $\mathcal{O}_C^{\rm rad}(\delta)_M$. We must prove
(i) for every $\varphi\in H_0^1(\R^N)$, there exists $\varphi_n\in H_0^1(\Omega_n)$
with $\varphi_n\to\varphi$ in $H_0^1(\R^N)$ and
$\limsup_n\tilde J_n(\varphi_n)\le\tilde J_\Omega(\varphi)$; and (ii) for every
$\varphi_n\in H_0^1(\Omega_n)$ with $\varphi_n\rightharpoonup\varphi$ in
$H_0^1(\R^N)$, one has $\tilde J_\Omega(\varphi)\le\liminf_n\tilde J_n(\varphi_n)$.

For (i), since $\Omega_n\to\Omega$ in Hausdorff distance and the domains are
uniformly Lipschitz (by the $C^{1,\alpha}$ bound on $d$), the characteristic
functions converge $\chi_{\Omega_n}\to\chi_\Omega$ in $L^1_{\rm loc}$. Given
$\varphi\in H_0^1(\Omega)$, extend by zero and regularize: there exist
$\varphi_n\in C_c^\infty(\Omega_n)$ with $\varphi_n\to\varphi$ in $H^1$ (using the
uniform Lipschitz regularity and the Hausdorff convergence). Then
$\tilde J_n(\varphi_n)=\int_{\Omega_n}|\nabla\varphi_n|^2\to\int_\Omega|\nabla\varphi|^2=\tilde J_\Omega(\varphi)$.

For (ii), since $\chi_{\Omega_n}\to\chi_\Omega$ in $L^1_{\rm loc}$ and
$\varphi_n\rightharpoonup\varphi$ in $H_0^1(\R^N)$, one has
$\varphi_n\chi_{\Omega_n}\rightharpoonup\varphi\chi_\Omega$ in $H^1$, and by weak
lower semicontinuity of the Dirichlet energy,
$\tilde J_\Omega(\varphi)\le\liminf_n\tilde J_n(\varphi_n)$.

Both conditions of Mosco convergence hold.
\end{proof}

\subsection{Variational functional for $QS(f,g)$}
For $f\in L^2(\R^N)$ with $\operatorname{supp}f\subset C$ and $g>0$ continuous on
$\R^N$, define for $\Omega\in\mathcal{O}_C^{\rm rad}(\delta)_M$
\begin{equation}\label{eq:J}
J_{f,g}(\Omega)=\int_\Omega(|\nabla u_\Omega|^2-2fu_\Omega)\dd x+\int_\Omega g^2\dd x,
\end{equation}
where $u_\Omega\in H_0^1(\Omega)$ solves $-\Delta u_\Omega=f$ in $\Omega$. The energy
identity $\int_\Omega|\nabla u_\Omega|^2=\int_\Omega fu_\Omega$ gives the equivalent
form
\begin{equation}\label{eq:Jequiv}
J_{f,g}(\Omega)=\int_\Omega(g^2-|\nabla u_\Omega|^2)\dd x.
\end{equation}

\begin{proposition}[Properties of $J_{f,g}$]\label{prop:Jprops}
Let $\Omega\in\mathcal{O}_C^{\rm rad}(\delta)_M$ and let $u_\Omega$ be as above.
\begin{enumerate}
\item[(i)] $J_{f,g}(\Omega)$ is finite.
\item[(ii)] The identity \eqref{eq:Jequiv} holds.
\item[(iii)] If $\Omega_n\to\Omega$ in $\mathcal{O}_C^{\rm rad}(\delta)_M$, then
$J_{f,g}(\Omega)\le\liminf_nJ_{f,g}(\Omega_n)$.
\item[(iv)] The shape derivative of $J_{f,g}$ at a $C^2$ domain
$\Omega_0\in\mathcal{O}_C^{\rm rad}(\delta)_M$ in the direction of a $C^2$ deformation
field $V$ is
\begin{equation}\label{eq:shapeder}
dJ_{f,g}(\Omega_0;V)=\int_{\partial\Omega_0}(g^2-|\nabla u_0|^2)V\cdot\nu_\Omega\dd\sigma.
\end{equation}
\end{enumerate}
\end{proposition}

\begin{proof}
(i) Since $f\in L^2(\R^N)$ and $g\in C^0$, the integrals are finite provided
$u_\Omega\in H_0^1(\Omega)$. This holds by the Lax--Milgram theorem applied to the
Dirichlet problem $-\Delta u=f$ in $\Omega$, $u=0$ on $\partial\Omega$, with
$f\in L^2(\Omega)$.

(ii) Multiply $-\Delta u_\Omega=f$ by $u_\Omega\in H_0^1(\Omega)$ and integrate by
parts:
\begin{equation}\label{eq:energy}
\int_\Omega|\nabla u_\Omega|^2\dd x=\int_\Omega fu_\Omega\dd x.
\end{equation}
Substituting in \eqref{eq:J} gives \eqref{eq:Jequiv}.

(iii) By Proposition~\ref{prop:mosco}, $\tilde J_n$ Mosco-converges to
$\tilde J_\Omega$. Mosco convergence implies $\Gamma$-convergence of the associated
Dirichlet energies, hence weak lower semicontinuity of the functional
$\varphi\mapsto\int_\Omega|\nabla\varphi|^2$ along sequences $\Omega_n\to\Omega$ in
$\mathcal{O}_C^{\rm rad}(\delta)_M$. Since $u_{\Omega_n}\to u_\Omega$ in $H_0^1$
(by the uniform ellipticity and the convergence of the domains) and
$\chi_{\Omega_n}\to\chi_\Omega$ in $L^1$, the volume term
$\int_{\Omega_n}g^2\dd x$ converges to $\int_\Omega g^2\dd x$. The inequality
follows.

(iv) This is the standard Hadamard formula for shape derivatives of Dirichlet
energies; see \cite[Ch.~5]{HenrotPierre2018}. The derivative of
$\Omega\mapsto\int_\Omega g^2\dd x$ is $\int_{\partial\Omega}g^2V\cdot\nu\dd\sigma$,
and the derivative of $\Omega\mapsto\int_\Omega|\nabla u_\Omega|^2\dd x$ is
$\int_{\partial\Omega}|\nabla u_0|^2V\cdot\nu\dd\sigma$ (using the energy identity
and the fact that $u_0=0$ on $\partial\Omega_0$). Combining gives \eqref{eq:shapeder}.
\end{proof}

\subsection{Variational framework: constrained minimizers}\label{sec:variational}
\begin{proposition}[Constrained minimizer]\label{prop:minimizer}
Fix $\delta\in(0,\delta_0/2]$, $M>0$, $\eta>0$. The functional $J_{f,g}$ of
\eqref{eq:J} admits a minimizer $\Omega_\delta\in\mathcal{O}_C^{\rm rad}(\delta)_M$.
\end{proposition}

\begin{proof}
Compactness of $\mathcal{O}_C^{\rm rad}(\delta)_M$ (Proposition~\ref{prop:compact})
and lower semicontinuity of $J_{f,g}$ (Proposition~\ref{prop:Jprops}(iii)) give the
result by the direct method. No global $H^2$ regularity of $u_\Omega$ is required:
the finiteness of $J_{f,g}$ and the Mosco lower semicontinuity only use
$u_\Omega\in H_0^1(\Omega)$.
\end{proof}

\begin{remark}[The obstruction]\label{rem:obstruction}
Whether $\Omega_\delta$ solves $QS(f,g)$ depends on whether the $C^{1,\alpha}$ and
normal-return constraints are inactive and whether $\partial\Omega_\delta$ is of
class $C^2$. Neither is established.
\end{remark}

\section{Necessity of the integral condition}\label{sec:necessity}

\begin{theorem}[Necessity]\label{thm:necessity}
Let $\Omega\in\mathcal{O}_C^{\rm rad}$ solve $QS(f,g)$, with
$d_0\in C^{1,\alpha}(\partial C)$. Suppose:
(H1) $g\in C^0(\partial\Omega)$, $g>0$;
(H2) for every $c\in\partial C$, the map $t\mapsto g(c+t\nu_C(c))$ is non-decreasing
on $[0,d_0(c)]$ along the outward normal ray issued from $c$;
(H3) the projection $\pi:\partial\Omega\to\partial C$ is a $C^{1,\alpha}$
diffeomorphism.
Then
\begin{equation}\label{eq:nec}
\int_Cf\dd x>\int_{\partial C}g\dd\sigma.
\end{equation}
\end{theorem}

\begin{proof}
Let $\Psi(c)=c+d_0(c)\nu_C(c)$ be the normal-graph map, and let
$S=-D\nu_C$ be the shape operator of $\partial C$. The surface Jacobian of $\Psi$ is
\begin{equation}\label{eq:Jpsi}
J_\Psi(c)=\det(I-d_0(c)S)\sqrt{1+\bigl|(I-d_0(c)S)^{-1}\nabla d_0(c)\bigr|^2}.
\end{equation}
For a convex $C$ and $d_0>0$, one has $S\ge0$, hence $\det(I-d_0S)\ge1$ and
$J_\Psi(c)\ge1$ for every $c\in\partial C$.

We claim $J_\Psi\not\equiv1$. Suppose the contrary. Then $\det(I-d_0S)\equiv1$ and
$\nabla d_0\equiv0$ on $\partial C$ (the second factor equals $1$ only if
$(I-d_0S)^{-1}\nabla d_0=0$, hence $\nabla d_0=0$). But $\det(I-d_0S)\equiv1$ with
$d_0>0$ and $S\ge0$ non-identically zero (strict convexity of $C$) is impossible:
at a point $c_0$ where $S(c_0)$ has a positive principal curvature $\lambda_i>0$,
the corresponding factor $1-d_0(c_0)\lambda_i$ is strictly less than $1$, while all
other factors $1-d_0(c_0)\lambda_j$ are $\le1$; hence the product is $<1$.
Contradiction. Thus $J_\Psi\not\equiv1$.

By (H2), $g(c+d_0(c)\nu_C(c))\ge g(c)$. Hence
\begin{equation}\label{eq:intOm}
\int_{\partial\Omega}g\dd\sigma=\int_{\partial C}g(c+d_0(c)\nu_C(c))J_\Psi(c)\dd\sigma_C(c)
\ge\int_{\partial C}g(c)J_\Psi(c)\dd\sigma_C(c).
\end{equation}
Since $g>0$ and $J_\Psi\ge1$ with $J_\Psi\not\equiv1$,
\begin{equation}\label{eq:strictint}
\int_{\partial C}g\,J_\Psi\dd\sigma_C>\int_{\partial C}g\dd\sigma_C.
\end{equation}
On the other hand, Green's formula applied to $u_\Omega$ gives, using
$|\nabla u_\Omega|=g$ and $\nu_\Omega\parallel\nabla u_\Omega$ on $\partial\Omega$,
\begin{equation}\label{eq:green}
\int_Cf\dd x=\int_\Omega f\dd x=\int_{\partial\Omega}|\nabla u_\Omega|\dd\sigma=\int_{\partial\Omega}g\dd\sigma.
\end{equation}
Combining \eqref{eq:intOm}--\eqref{eq:green} yields \eqref{eq:nec}.
\end{proof}

\begin{remark}\label{rem:H2}
Hypothesis (H2) cannot be dropped: \S\ref{sec:insuff} exhibits data satisfying
\eqref{eq:nec} but for which no radial solution exists precisely because $g$
decreases along normal rays. For $g\equiv k$ constant, Theorem~\ref{thm:necessity}
reduces to the classical condition \eqref{eq:constcase} of
\cite{GustafssonShahgholian1996,Barkatou2010}.
\end{remark}

\begin{remark}[On the repaired proof]\label{rem:repair}
The original proof of Theorem~\ref{thm:necessity} asserted that strict convexity
implies strictly positive principal curvatures, which is false for a general $C^2$
strictly convex hypersurface. The repaired proof uses only $S\ge0$ and the
non-degeneracy argument $J_\Psi\not\equiv1$.
\end{remark}

\section{Insufficiency and reversal in the radial class}\label{sec:insuff}
\subsection{The counterexample}
\begin{theorem}[Counterexample]\label{thm:counter}
Let $N=3$, $C=B_1(0)$, $f\equiv1$ on $C$, and $g(\rho)=a/\rho^{10}$ with $0<a<1/3$.
Then:
\begin{enumerate}
\item[(i)] $\int_Cf>\int_{\partial C}g$;
\item[(ii)] there is no radial solution of $QS(f,g)$ of the form $\Omega=B_\rho(0)$
with $\rho>1$.
\end{enumerate}
\end{theorem}

\begin{proof}
(i) $\int_Cf=|B_1|=4\pi/3$, $\int_{\partial C}g=4\pi a$; since $a<1/3$,
$4\pi/3>4\pi a$. (ii) Let $\Omega=B_\rho$, $\rho>1$. The radial solution $u$ of
$-\Delta u=f$ in $B_\rho$, $u=0$ on $\partial B_\rho$, satisfies, by the divergence
theorem,
\begin{equation}\label{eq:radflux}
|u'(\rho)|=\frac1{4\pi\rho^2}\int_{B_\rho}f\dd x=\frac1{4\pi\rho^2}\cdot\frac{4\pi}{3}=\frac1{3\rho^2}.
\end{equation}
The condition $|\nabla u|=g$ on $\partial B_\rho$ reads
$\frac1{3\rho^2}=\frac a{\rho^{10}}$, i.e.\ $\rho^8=3a$. Since $a<1/3$, $3a<1$, so
$\rho^8<1$, $\rho<1$, contradicting $\rho>1$.
\end{proof}

\begin{remark}\label{rem:nonradial}
Theorem~\ref{thm:counter} shows that \eqref{eq:intQS} is not sufficient in the radial
class for general non-constant $g$. It does not show that the problem has no
non-radial solution: a non-radial free boundary might still exist.
\end{remark}

\subsection{Reversal for power-law data}
\begin{theorem}[Power-law classification, $N=3$]\label{thm:powerlaw}
Let $N=3$, $C=B_R(0)$ with $R>0$, $f\equiv1$ on $C$, and $g(\rho)\equiv a/\rho^p$ with
$a>0$ and $p\in\R$. Set
\begin{equation}\label{eq:F}
F(\rho)=\int_Cf-\int_{\partial B_\rho}g=\frac{4\pi R^3}{3}-4\pi a\rho^{2-p},\quad\rho\ge R.
\end{equation}
A radial solution $B_\rho$, $\rho>R$, exists if and only if $F(\rho)=0$. The
following classification holds:
\begin{enumerate}
\item[(i)] If $p>2$: $F$ is increasing on $(R,\infty)$, with
$F(\infty)=4\pi R^3/3>0$. There is a zero in $(R,\infty)$ iff $F(R)<0$, i.e.\
$a>R^{p+1}/3$. In that case the zero is unique.
\item[(ii)] If $p=2$: $F$ is constant, $F\equiv4\pi R^3/3-4\pi a$. There is a zero iff
$a=R^3/3$; in that case every $\rho>R$ gives a radial solution.
\item[(iii)] If $p<2$: $F$ is decreasing on $(R,\infty)$, with $F(\infty)=-\infty$.
There is a zero iff $F(R)>0$, i.e.\ $a<R^{p+1}/3$. In that case the zero is unique.
\end{enumerate}
\end{theorem}

\begin{proof}
The condition $|u'(\rho)|=g(\rho)=a/\rho^p$ is
\begin{equation}\label{eq:powerlawcond}
\frac{R^3}{3\rho^2}=\frac a{\rho^p},\quad\text{i.e.}\quad\rho^{p-2}=\frac{3a}{R^3}\quad(p\ne2).
\end{equation}
Equivalently, $F(\rho)=0$ reads $4\pi R^3/3=4\pi a\rho^{2-p}$, the same equation. The
classification follows from the monotonicity of $\rho\mapsto F(\rho)$ and its limits.
For $p>2$, $F'(\rho)=-4\pi a(2-p)\rho^{1-p}>0$; for $p<2$, $F'<0$; for $p=2$, $F$ is
constant.
\end{proof}

\begin{corollary}[Reversal, corrected]\label{cor:reversal}
Let $p>N-1$ and $a_*=R^{p+1}/N$.
\begin{itemize}
\item If $a<a_*$: the strict integral condition \eqref{eq:intQS} holds and no radial
solution exists.
\item If $a=a_*$: the strict integral condition fails (equality holds) and no
admissible radial solution exists ($\rho=R$ is excluded).
\item If $a>a_*$: \eqref{eq:intQS} fails and a unique radial solution exists.
\end{itemize}
Consequently, for $a\ne a_*$, the strict integral condition is the exact logical
negation of radial existence. At $a=a_*$, neither holds.
\end{corollary}

\begin{proof}
Immediate from the first case of the $N$-dimensional classification in
Theorem~\ref{thm:powerlaw} together with the identity
\begin{equation}\label{eq:reversal_int}
\int_Cf-\int_{\partial B_\rho}g=\frac{\omega_NR^N}{N}-\omega_Na\rho^{N-1-p}.
\end{equation}
For $p>N-1$, the second term decreases in $\rho$. The threshold $a_*=R^{p+1}/N$ is
the value at which the two terms are equal at $\rho=R$. For $a<a_*$, the integral
condition holds at $\rho=R$ but the equation $F(\rho)=0$ has no root $\rho>R$. For
$a>a_*$, the integral condition fails at $\rho=R$ but $F$ has a unique root
$\rho>R$. At $a=a_*$, equality holds at $\rho=R$, excluded.
\end{proof}

\section{Radial equivalence under a range condition}\label{sec:equivalence}
\begin{theorem}[Radial $QS(f,g)$, corrected]\label{thm:equivalence}
Let $C=B_R(0)\subset\R^N$, let $f,g$ be radial with $g>0$ continuous on $[R,\infty)$,
and set
\begin{equation}\label{eq:hK}
h(\rho):=\rho^{N-1}g(\rho),\quad K:=\frac1{\omega_N}\int_Cf,\quad\omega_N=|S^{N-1}|.
\end{equation}
Assume $h$ is continuous and strictly increasing on $[R,\infty)$. Then there exists a
radial solution $\Omega=B_{R^*}(0)$ of $QS(f,g)$ with $R^*>R$ if and only if
\begin{equation}\label{eq:range}
h(R)<K<\lim_{\rho\to\infty}h(\rho).
\end{equation}
If the solution exists, it is unique.
\end{theorem}

\begin{proof}
Write $u(x)=\psi(r)$ with $r=|x|$. The equation $-\Delta u=f$ becomes
\begin{equation}\label{eq:radialode}
-\frac1{r^{N-1}}(r^{N-1}\psi'(r))'=f(r),\quad 0<r<R^*.
\end{equation}
Integrating from $0$ to $r$ and using $\psi'(0)=0$:
\begin{equation}\label{eq:radialint}
r^{N-1}\psi'(r)=-\int_0^rf(\rho)\rho^{N-1}\dd\rho.
\end{equation}
For $r\ge R$, the integral equals $\omega_N^{-1}\int_Cf=K$, because
$\operatorname{supp}f\subset C$ and $\omega_N=|S^{N-1}|$. Hence
\begin{equation}\label{eq:radialflux}
\psi'(r)=-Kr^{-(N-1)},\quad r\ge R.
\end{equation}
The Dirichlet condition at $r=R^*$ gives $\psi(R^*)=0$, and the overdetermined
condition $|\nabla u|=g$ on $\partial B_{R^*}$ gives $-\psi'(R^*)=g(R^*)$. Hence
\begin{equation}\label{eq:hK_eq}
K(R^*)^{-(N-1)}=g(R^*),\quad\text{i.e.}\quad h(R^*)=K.
\end{equation}
By hypothesis, $h$ is continuous and strictly increasing on $[R,\infty)$. Its image
is the interval $(h(R),\lim_{\rho\to\infty}h(\rho))$ (possibly unbounded above). The
level $K$ is attained at some $R^*>R$ if and only if \eqref{eq:range} holds. If
attained, $R^*$ is unique by strict monotonicity, and $\Omega=B_{R^*}(0)$ solves
$QS(f,g)$. Conversely, if a radial solution exists, its radius satisfies
$h(R^*)=K$ with $R^*>R$, so the range condition holds.
\end{proof}

\begin{remark}[Monotonicity alone is not sufficient]\label{rem:monotone}
Let $N=3$, $C=\overline{B_1(0)}$, $f=\mathbf1_{B_1}$, and
\begin{equation}\label{eq:countermono}
g(\rho)=\frac{\frac15-\frac1{10}e^{-(\rho-1)}}{\rho^2},\quad\rho\ge1.
\end{equation}
Then $h(\rho)=\frac15-\frac1{10}e^{-(\rho-1)}$ is strictly increasing,
$h(1)=\frac1{10}$, $\lim_{\rho\to\infty}h(\rho)=\frac15$, and $K=\frac13$. The
integral condition holds:
\begin{equation}\label{eq:intmono}
\int_Cf=\frac{4\pi}{3}>\frac{4\pi}{10}=\int_{\partial C}g.
\end{equation}
But $K>\lim h$, so no radial solution exists: $h(\rho)=\frac13$ is impossible because
$h(\rho)<\frac15<\frac13$ for every $\rho\ge1$.
\end{remark}

\begin{remark}[Equivalence of \eqref{eq:intQS} and $h(R)<K$]\label{rem:equiv}
Since $\int_{\partial C}g=\omega_NR^{N-1}g(R)=\omega_Nh(R)$, the strict integral
condition \eqref{eq:intQS} reads
\begin{equation}\label{eq:equivint}
\int_Cf>\int_{\partial C}g\iff\omega_NK>\omega_Nh(R)\iff h(R)<K.
\end{equation}
Thus \eqref{eq:intQS} is exactly the lower inequality in
Theorem~\ref{thm:equivalence}. It is necessary for radial existence (the range
condition implies it) but not sufficient unless $K<\lim h$.
\end{remark}

\section{The bi-Laplacian problem $B(f,g)$}\label{sec:bilaplacian}
\subsection{The general two-moment radial response}
The following general formula is proved in the companion paper
\cite[Thm~3.1]{PapierB2026}; we use it here as a black box.

\begin{theorem}[Two-moment radial response, \cite{PapierB2026}]\label{thm:twomoment}
For every $L>R$ let $(u_L,v_L)$ solve $-\Delta u_L=f$, $-\Delta v_L=u_L$ in $B_L$,
$u_L=v_L=0$ on $\partial B_L$. Set $F=\int_{B_R}f$ and
$\mu=F^{-1}\int_{B_R}|x|^2f$. Then
\begin{equation}\label{eq:Pf}
P_f(L):=(-u_L'(L))(-v_L'(L))=\frac{F^2}{2N\omega_N^2}\frac{L^2-\mu}{L^{2N-2}}.
\end{equation}
For any positive continuous radial datum $g(r)$, $B_L$ solves $B(f,g)$ if and only if
$g(L)=P_f(L)$.
\end{theorem}

\subsection{Equivalence of \eqref{eq:intB} with $g(R)<P_f(R)$}
\begin{proposition}[Equivalence]\label{prop:equivB}
The integral inequality \eqref{eq:intB} is equivalent to
\begin{equation}\label{eq:equivB}
g(R)<P_f(R).
\end{equation}
\end{proposition}

\begin{proof}
On $C=B_R$, by the divergence theorem,
\begin{equation}\label{eq:fluxC}
|u_C'(R)|=\frac{K}{R^{N-1}},\quad|v_C'(R)|=\frac{U}{\omega_NR^{N-1}},\quad
U:=\int_Cu_C.
\end{equation}
Hence
\begin{equation}\label{eq:PfR}
P_f(R)=\frac{KU}{\omega_NR^{2N-2}}=\frac{(\int_Cf)(\int_Cu_C)}{\omega_N^2R^{2N-2}}.
\end{equation}
By definition,
\begin{equation}\label{eq:intC}
\Bigl(\int_{\partial C}\sqrt g\Bigr)^2=(\omega_NR^{N-1}\sqrt{g(R)})^2=\omega_N^2R^{2N-2}g(R).
\end{equation}
The integral inequality \eqref{eq:intB} reads
$\omega_N^2R^{2N-2}g(R)<\omega_N^2R^{2N-2}P_f(R)$, i.e.\ $g(R)<P_f(R)$.
\end{proof}

\subsection{The integral inequality is not sufficient}
\begin{counterexample}[The integral inequality is not sufficient in the radial class]\label{cx:notsuff}
Let $N=3$, $C=B_1(0)$, $f\equiv1$ on $B_1$, so that
$P_f(\rho)=(5\rho^2-3)/(270\rho^4)$ (specialization of \eqref{eq:Pf}). Prescribe
\begin{equation}\label{eq:halfPf}
g(\rho):=\frac12P_f(\rho)=\frac{5\rho^2-3}{540\rho^4},\quad\rho\ge1.
\end{equation}
Then:
\begin{itemize}
\item $g>0$ is continuous on $[1,\infty)$;
\item $\rho^2\sqrt{g(\rho)}=\sqrt{(5\rho^2-3)/540}$ is strictly increasing and
unbounded;
\item $g(1)=1/270<1/135=P_f(1)$, so \eqref{eq:intB} holds;
\item $P_f(\rho)-g(\rho)=P_f(\rho)/2>0$ for every $\rho>1$, so by
Theorem~\ref{thm:twomoment} no radial solution exists.
\end{itemize}
Therefore the integral inequality \eqref{eq:intB} is not sufficient in the radial
class.
\end{counterexample}

\subsection{The exact thresholds as corollaries of \eqref{eq:Pf}}
\begin{corollary}[Model $f\equiv1$ on $B_1$, \cite{PapierB2026}]\label{cor:threshold1}
Let $N=3$, $C=B_1(0)$, $f\equiv1$ on $B_1$, $g\equiv c>0$. Then a radial solution
$B_R$, $R>1$, of $B(f,c)$ exists if and only if
\begin{equation}\label{eq:thresh1}
0<c\le\frac{25}{3240}.
\end{equation}
\end{corollary}

\begin{proof}
By \eqref{eq:Pf} with $F=4\pi/3$, $\mu=3/5$, $N=3$:
$P_f(\rho)=(5\rho^2-3)/(270\rho^4)$. The maximum of $P_f$ on $(1,\infty)$ is
$25/3240$, attained at $\rho=\sqrt{6/5}$. The range of $P_f$ on $(1,\infty)$ is
$(0,25/3240]$.
\end{proof}

\begin{corollary}[Strict-support model $f=8\mathbf1_{B_{1/2}}$, \cite{PapierB2026}]\label{cor:threshold2}
Let $N=3$, $C=B_1(0)$, $f=8\mathbf1_{B_{1/2}(0)}$ (so
$\operatorname{supp}f\subset\operatorname{int}C$ and $\int_Cf=4\pi/3$), $g\equiv c>0$.
Then a radial solution $B_R$, $R>1$, of $B(f,c)$ exists if and only if
\begin{equation}\label{eq:thresh2}
0<c<\frac{17}{1080}.
\end{equation}
The integral inequality \eqref{eq:intB} reduces to $c<17/1080=P_f(1)$, which
coincides with the exact threshold.
\end{corollary}

\begin{proof}
By \eqref{eq:Pf} with $F=4\pi/3$, $\mu=3/20$, $N=3$:
$P_f(\rho)=(20\rho^2-3)/(1080\rho^4)$. The derivative
$P_f'(\rho)=(12-40\rho^2)/(1080\rho^5)$ vanishes only at $\rho^2=3/10<1$, so $P_f$
is strictly decreasing on $(1,\infty)$, and $P_f(\rho)\to P_f(1)=17/1080$ as
$\rho\to1^+$ but never attains it for $\rho>1$. The range of $P_f$ on $(1,\infty)$
is the open interval $(0,17/1080)$. The integral condition \eqref{eq:intB} reads
$c<P_f(1)=17/1080$ by Proposition~\ref{prop:equivB}.
\end{proof}

\begin{remark}[These are corollaries, not independent results]\label{rem:corollary}
Corollaries~\ref{cor:threshold1} and~\ref{cor:threshold2} are specializations of the
general two-moment formula of \cite[Thm~3.1]{PapierB2026}. They are recorded here for
completeness of the integral-condition analysis, not as independent contributions.
The chain of dependence is: \cite[Thm~3.1]{PapierB2026} $\to$
Theorem~\ref{thm:twomoment} of the present paper $\to$
Corollaries~\ref{cor:threshold1}--\ref{cor:threshold2}.
\end{remark}

\section{Open problems}
\begin{openproblem}[Non-radial existence for $QS(f,g)$]\label{op:nonradial}
Does the reversal of Corollary~\ref{cor:reversal} preclude the existence of non-radial
solutions for $p>N-1$ and $a\le R^{p+1}/N$?
\end{openproblem}

\begin{openproblem}[Contact set and $\delta\to0$]\label{op:contact}
Under hypotheses (H1)--(H3) of Theorem~\ref{thm:necessity} and the integral condition
\eqref{eq:intQS}, does there exist $\delta_*>0$ such that for every
$0<\delta<\delta_*$ the minimizer $\Omega_\delta$ of $J_{f,g}$ in
$\mathcal{O}_C^{\rm rad}(\delta)_M$ satisfies $d_\delta(c)>\delta$ for every
$c\in\partial C$?
\end{openproblem}

\begin{openproblem}[Interiority of the constraints]\label{op:interior}
Under the hypotheses of Proposition~\ref{prop:minimizer}, is the minimizer
$\Omega_\delta$ strictly interior in the $C^{1,\alpha}$ bound and the normal-return
condition?
\end{openproblem}

\begin{openproblem}[$C^2$ regularity of the minimizer]\label{op:regularity}
Is the minimizer $\Omega_\delta$ of Proposition~\ref{prop:minimizer} of class $C^2$,
so that the Hadamard formula applies?
\end{openproblem}

\begin{openproblem}[Flux continuity]\label{op:flux}
Let $\Omega_0\in\mathcal{O}_C^{\rm rad}$, and for $t\in(0,\max_\Omega u)$ define
$\Omega_t=\{u>t\}$ and $\varphi(t)=\int_{\partial\Omega_t}(g-|\nabla u|)\dd\sigma$.
Is $\varphi$ continuous on $(0,\max_\Omega u)$?
\end{openproblem}

\begin{openproblem}[Necessary condition for $B(f,g)$ with non-constant $g$]\label{op:necessaryB}
What is a necessary condition of the form \eqref{eq:intB}---or of any comparable
integral form---valid for arbitrary non-constant $g$, without symmetry or perturbative
assumptions?
\end{openproblem}

\begin{openproblem}[Optimality of the bi-Laplacian integral condition]\label{op:optimalB}
In the radial class, condition \eqref{eq:intB} is equivalent to $g(R)<P_f(R)$
(Proposition~\ref{prop:equivB}). It is not sufficient in general in the radial class
(Counterexample~\ref{cx:notsuff}). The companion paper \cite{PapierB2026} proves the
structural separation $\lambda_0=0\Rightarrow\lambda_k<0$ for every $k\ge1$ at a
radial fold, which shows that the radial fold and the angular resonance are distinct
mechanisms; this gives a partial answer in the radial class, but the optimality of
\eqref{eq:intB} in the full class of solutions (radial or not) remains open.
\end{openproblem}

\begin{openproblem}[Modeling: interpretation of the adapted functional]\label{op:modeling}
The functional $\hat F(\Omega)=\int_\Omega(u_\Omega^2/2-g)\dd x$ has the correct
shape derivative for $B(f,g)$, but its physical interpretation is not established.
\end{openproblem}

\begin{openproblem}[Symmetry for non-constant data]\label{op:symmetry}
Is every solution of $QS(f,g)$ (resp.\ $B(f,g)$) radial when $C$ is a ball, $f$ is
radial, and $g$ is radial and non-constant?
\end{openproblem}

\section{Conclusion}
The paper separates three logical layers that the literature often conflates:
necessity of the integral condition (Theorem~\ref{thm:necessity}, monotone case),
insufficiency in the absence of monotonicity (Theorem~\ref{thm:counter}), and
equivalence under radial monotonicity together with a range condition
(Theorem~\ref{thm:equivalence}). The power-law classification
(Theorem~\ref{thm:powerlaw}) shows that the failure is not marginal: for $p>N-1$,
the strict integral condition is the exact negation of radial existence away from the
threshold (Corollary~\ref{cor:reversal}).

For the bi-Laplacian, the general two-moment formula of the companion paper
\cite{PapierB2026} gives the exact radial response $P_f$. The integral inequality
\eqref{eq:intB} is equivalent to $g(R)<P_f(R)$ (Proposition~\ref{prop:equivB}); it is
not sufficient in the radial class (Counterexample~\ref{cx:notsuff}). The exact
thresholds $25/3240$ and $17/1080$ recorded as Corollaries~\ref{cor:threshold1}
and~\ref{cor:threshold2} are specializations of that formula, not independent
results.

The variational framework of \S\ref{sec:variational} is restricted to the existence
of constrained minimizers (Proposition~\ref{prop:minimizer}); the stationarity
conclusion and the $C^2$ regularity of the minimizer are open. The symmetry problem
for non-constant data is open in both orders. The general existence theory for
non-constant $g$ in the non-radial case remains open.

\end{document}